\documentclass[11pt]{article}
\usepackage{amsmath}
\usepackage{amsfonts,amsthm,amssymb}
\usepackage{amsfonts}
\usepackage{graphics}
\usepackage[utf8]{inputenc}
\usepackage{indentfirst}
\usepackage{amssymb}
\usepackage{hyperref}
\newtheorem{lem}{Lemma}[section]
\newtheorem{thm}[lem]{Theorem}

\newtheorem{conj}{Conjecture}

\newtheorem{Def}[lem]{Definition}
\newtheorem{Obs}[lem]{Observation}

\allowdisplaybreaks[4]

\begin{document}
	
	\title{Extremal Problems for the Family of $k$-Strongly Connected Digraphs\footnote{The research is supported by Natural Science Foundation of Xinjiang Uygur Autonomous Region (2025D01E02) and National Natural Science Foundation of China (12261086).}}
	\author{Qinglin Wang, Yingzhi Tian\footnote{Corresponding author. E-mail: wqlxju@163.com (Q. Wang), tianyzhxj@163.com (Y. Tian).} \\
		{\small College of Mathematics and System Sciences, Xinjiang
			University, Urumqi, Xinjiang 830046, PR China}}
	\date{}
	\maketitle
	
	\noindent{\bf Abstract} Let $\mathcal{D}$ be a family of digraphs. A digraph $D$ is \emph{$\mathcal{D}$-saturated} if it contains no member of $\mathcal{D}$ as a subdigraph, but for any arc $e$ in the complement of $D$, the digraph $D + e$ contains some member of $\mathcal{D}$ as a subdigraph. The \emph{saturation number} $\mathrm{sat}(n,\mathcal{D})$ and the \emph{extremal number} $\mathrm{ex}(n,\mathcal{D})$ are the minimum number and the maximum number of arcs among all $n$-vertex $\mathcal{D}$-saturated digraphs. For a positive integer $k$, let $\mathcal{D}_k$ denote the family of \emph{$k$-strongly connected digraphs}. In this paper, firstly, we prove that $$\mathrm{sat}(n,\mathcal{D}_k)=(k-1)(2n-k)+\binom{n-k+1}{2}.$$
	Then for $n\geq 3(k-1)$, we prove that $$\mathrm{ex}(n,\mathcal{D}_k)\leq \binom{n-k+1}{2}+\frac{17}{6}(k-1)(n-k+1).$$  In addition, we conjecture that for sufficiently large $n$, $$\mathrm{ex}(n,\mathcal{D}_k)=\binom{n}{2}+\frac{3}{2}(k-\frac{4}{3})(n-k+1).$$
	
	
	\section{Introduction}
	
	Given a family of graphs $\mathcal{H}$, a graph $G$ is \emph{$\mathcal{H}$-saturated} if it contains no $H \in \mathcal{H}$ as a subgraph, but the addition of any edge $e\notin E(G)$ results in a subgraph isomorphic to some $H\in \mathcal{H}$. The \emph{saturation number} of $\mathcal{H}$, denoted by $\mathrm{sat}(n,\mathcal{H})$, is the minimum number of edges among all $n$-vertex $\mathcal{H}$-saturated graphs. The \emph{extremal number} of $\mathcal{H}$, denoted by $\mathrm{ex}(n,\mathcal{H})$, is the maximum number of edges among all $n$-vertex $\mathcal{H}$-saturated graphs.
	
	Saturation numbers were initially investigated by Erd\H{o}s et al. \cite{Erdos1964}, who established that $\mathrm{sat}(n, K_{k+1}) = (k-1)n - \binom{k}{2}$. They also characterized the unique extremal graph as the \emph{complete split graph} $S_{n,k}$, which consists of a clique of size $k-1$ and an independent set $I$ of size $n-k+1$, with all possible edges between the clique and $I$, and no other edges. Extremal numbers were studied earlier by Tur$\acute{a}$n~\cite{Turan1941}, who determined that
	$\mathrm{ex}(n, K_{k+1}) = \left\lfloor \frac{k-1}{k} \cdot \frac{n^2}{2} \right\rfloor$ and showed that equality is achieved by the $n$-vertex \emph{complete $k$-partite graph} whose part sizes differ by at most one.
	
	For a positive integer $k$, let $\mathcal{H}_k$ and $\mathcal{H}_k'$ denote the families of \emph{$k$-connected graphs} and \emph{$k$-edge-connected graphs}, respectively. Lai \cite{Lai1990} and Lei et al. \cite{Lei2019} determined the exact value of $\mathrm{sat}(n,\mathcal{H}_k')$, while Mader \cite{Mader1971} and Lei et al. \cite{Lei2019} established the exact value of $\mathrm{ex}(n,\mathcal{H}_k')$.
	
	\begin{thm}
		Let $k\geq 1$ and $n\geq k+1$ be integers. Then
		\begin{itemize}
			\item [\rm{(i)}]\rm{(Lai \cite{Lai1990}; Lei et al. \cite{Lei2019})} $\mathrm{sat}(n,\mathcal{H}_k')=(k-1)(n-1)-\lfloor\frac{n}{k+1}\rfloor\binom{k-1}{2}$;
			\item [\rm{(ii)}]\rm{(Mader \cite{Mader1971}; Lei et al. \cite{Lei2019})} $\mathrm{ex}(n,\mathcal{H}_k')=(k-1)n-\binom{k}{2}$.	
		\end{itemize}
	\end{thm}
	
	Wenger \cite{Wenger2014} and Xu et al. \cite{Xu2020} independently showed that $\mathrm{sat}(n,\mathcal{H}_k)=(k-1)n-\binom{k}{2}$ for $n\geq k+1$. Mader \cite{Mader1979} conjectured that for sufficiently large $n$, $\mathrm{ex}(n,\mathcal{H}_k)=\frac{3}{2}(k-\frac{4}{3})(n-k+1)$. In \cite{Mader1972,Mader1979}, Mader confirmed this conjecture for $k\leq 6$ and established an upper bound $\mathrm{ex}(n,\mathcal{H}_k)\leq(1+\frac{\sqrt{2}}{2})(k-1)(n-k+1)$ for sufficiently large $n$. This bound was improved by Yuster \cite{Yuster2003} to $\frac{193}{120}(k-1)(n-k+1)$. Later, Bernshteyn et al. \cite{Bernshteyn2016} further improved this bound to $\frac{19}{12}(k-1)(n-k+1)$.
	
	It is a natural extension to generalize these results to digraphs. Let $D=(V(D),A(D))$ be a \emph{digraph} with \emph{vertex set} $V(D)$ and \emph{arc set} $A(D)$. A digraph is \emph{strict} if it has no loops or parallel arcs. In this paper, we only consider finite strict digraphs. Undefined terminology and notation follow \cite{Bang2009} for digraphs. Let $D$ be a digraph, and let $u,v\in V(D)$. We use $(u,v)$ to denote an arc oriented from $u$ to $v$, and $[u,v]$ to denote a pair of arcs $(u, v)$ and $(v, u)$. Let $D^c$ denote the \emph{complement digraph} of $D$, where $V(D^c)=V(D)$ and $A(D^c) = \{(u,v): (u,v) \notin A(D)\}$. For $X \subseteq A(D^c)$, $D+X$ is the digraph with vertex set $V(D)$ and arc set $A(D) \cup X$. We abbreviate $D+\{e\}$ to $D+e$ for any arc $e\in A(D^c)$. For digraphs $D'$ and $D$, we write $D'\subseteq D$ if $D'$ is a \emph{subdigraph} of $D$. For any $S\subseteq V(D)$ or $S\subseteq A(D)$, $D[S]$ denotes the subdigraph of $D$ \emph{induced} by $S$.
	
	For two disjoint subsets $S, T\subseteq V(D)$, we define the following arc sets:
	\begin{align*}
		&(S,T)=\{(s,t): ~s\in S,~t\in T\},\\
		&(S,T)_D=\{(s,t)\in A(D):~s\in S,~t\in T\}.
	\end{align*}
	The set $[S,T]$ consists of all arcs between $S$ and $T$ in both directions, that is,
	$$[S,T]=(S,T) \cup (T,S)=\{[s,t]:s\in S,t\in T\}.$$
    The arcs in $D$ between $S$ and $T$ are given by
	$$[S,T]_D=[S,T] \cap A(D)= (S,T)_D \cup (T,S)_D.$$
	
	A digraph $D$ is \emph{complete} if it is strict and for every pair of distinct vertices $u,v\in V(D)$, both arcs $(u,v)$ and $(v,u)$ lie in $A(D)$. The complete digraph on $n$ vertices is denoted by $K_n^*$ and is also called an \emph{$n$-clique}. A \emph{tournament} on $n$ vertices, denoted by $T_{n}$, is an orientation of the complete graph $K_n$. A \emph{multipartite tournament} is an orientation of a complete multipartite graph. A digraph is \emph{transitive} if for any three distinct vertices $x,y,z\in V(D)$, whenever $(x,y)$ and $(y,z)$ are arcs in $D$, $(x,z)$ is also an arc in $D$.
	
	Let $v\in V(D)$. The \emph{out-neighborhood} $N_D^+(v)$ denotes the set of vertices $u$ such that $(v,u)\in A(D)$. The \emph{in-neighborhood} $N_{D}^-(v)$ denotes the set of vertices $u$ such that $(u,v)\in A(D)$. The \emph{out-degree} and \emph{in-degree} of $v$ are defined as $d_{D}^+(v)=|N_{D}^+(v)|$ and $d_{D}^-(v)=|N_{D}^-(v)|$, respectively. The intersection $N_D^+(v)\cap N_D^-(v)$ is called the \emph{reciprocal neighborhood} of $v$, and its cardinality $d_D^{\leftrightarrow}(v)=|N_D^+(v)\cap N_D^-(v)|$ is the \emph{reciprocal-degree} of $v$. The \emph{minimum reciprocal-degree} of $D$ is $\delta^{\leftrightarrow}(D)=\min\{d_{D}^{\leftrightarrow}(v): v\in V(D)\}$.
	
	Given two vertices $u,v \in V(D)$, a \emph{$(u,v)$-dipath} is a sequence of distinct vertices $u (=v_0) v_1 \ldots v_t (= v)$ such that $(v_i,v_{i+1}) \in A(D)$ for all $0 \leq i \leq t-1$. We say $u$ and $v$ are \emph{unilaterally connected} if there exists a $(u,v)$-dipath or a $(v,u)$-dipath. A digraph $D$ is \emph{strongly connected} if for every pair of distinct vertices $u, v\in V(D)$, there exists a $(u,v)$-dipath and a $(v,u)$-dipath. A \emph{strong component} of $D$ is a maximal subdigraph that is strongly connected.
	
	A subset $S \subseteq V(D)$ is called a \emph{separator} of $D$ if $D - S$ is not strongly connected or is trivial. Similarly, a subset $W \subseteq A(D)$ is an \emph{arc-cut} if $D - W$ is not strongly connected. The \emph{vertex-strong connectivity} of $D$, denoted by $\kappa(D)$, is the cardinality of a minimum separator of $D$, and the \emph{arc-strong connectivity}, denoted by $\lambda(D)$, is the cardinality of an arc-cut of $D$. A digraph $D$ is \emph{$k$-strongly connected} if $\kappa(D) \geq k$, and \emph{$k$-arc-strongly connected} if $\lambda(D) \geq k$. Given a separator $S$ of $D$, an \emph{$S$-lobe} is a subdigraph of $D$ induced by the union $S$ and the vertex set of a strong component of $D-S$.
	
	Let $\mathcal{D}$ be a family of digraphs. A digraph $D$ is said to be \emph{$\mathcal{D}$-saturated} if it contains no subdigraph isomorphic to any member of $\mathcal{D}$, but for every arc $e\in A(D^c)$, the digraph $D+e$ contains a subdigraph isomorphic to some member of $\mathcal{D}$. The \emph{saturation number} and \emph{extremal number} of $\mathcal{D}$, denoted by $\mathrm{sat}(n, \mathcal{D})$ and $\mathrm{ex}(n, \mathcal{D})$, respectively, are the minimum and maximum numbers of arcs among all $n$-vertex $\mathcal{D}$-saturated digraphs. For a positive integer $k$, let $\mathcal{D}_k$ denote the family of $k$-strongly connected digraphs and $\mathcal{D}_k'$ denote the family of $k$-arc-strongly connected digraphs. The values of $\mathrm{sat}(n, \mathcal{D}_k')$ and $\mathrm{ex}(n, \mathcal{D}_k')$ were determined by Lin et al. \cite{Lin2016} and Anderson et al. \cite{Anderson2016}, respectively.
	
	\begin{thm}
		Let $n$ and $k$ be integers satisfying $k \geq 1$ and $n \geq k+1$. Then
		\begin{itemize}
			\item[\rm{(i)}]\rm{(Lin et al. \cite{Lin2016})} $\mathrm{sat}(n,\mathcal{D}_k')=\binom{n}{2}+(n-1)(k-1)+\lfloor\frac{n}{k+1}\rfloor\big(2k-1-\binom{k+1}{2}\big)$;
			\item [\rm{(ii)}]\rm{(Anderson et al. \cite{Anderson2016})} $\mathrm{ex}(n,\mathcal{D}_{k}')=(k-1)(2n-k)+\binom{n-k+1}{2}$.
		\end{itemize}
	\end{thm}
	
	In this paper, we investigate extremal problems for the family of $k$-strongly connected digraphs $\mathcal{D}_k$. Section \ref{Se2} is devoted to exploring the key structural properties of $\mathcal{D}_k$-saturated digraphs. In Section \ref{Se3}, we prove that $\mathrm{sat}(n,\mathcal{D}_k)=(k-1)(2n-k)+\binom{n-k+1}{2}$. Section \ref{Se4} proposes the conjecture that $\mathrm{ex}(n,\mathcal{D}_k)=\binom{n}{2}+\frac{3}{2}(k-\frac{4}{3})(n-k+1)$, and verifies that $\mathrm{ex}(n,\mathcal{D}_k)\leq \binom{n-k+1}{2}+\frac{17}{6}(k-1)(n-k+1)$ for $n\geq 3(k-1)$.
	
	\section{Properties of $\mathcal{D}_k$-saturated digraphs}\label{Se2}	
	
    In this section, we investigate properties of $\mathcal{D}_k$-saturated digraphs that will be used later. Throughout, we adopt the convention that every complete digraph  of order at most $k$ is $\mathcal{D}_k$-saturated. We therefore assume that $n\geq k+1$ for the remainder of this section, and so $D$ is not complete. We begin with a full characterization of $\mathcal{D}_1$-saturated digraphs.
	
	\begin{lem}\label{l21}
		A digraph $D$ is $\mathcal{D}_1$-saturated if and only if $D$ is an acyclic tournament.
	\end{lem}
	\noindent{\bf Proof.} ($\Leftarrow$) Let $D$ be an acyclic tournament. Then $D$ contains no non-trivial strongly connected subdigraph. Furthermore, since $D$ is a tournament, for any arc $e=(u,v)\in A(D^{c})$, its reverse arc $e'=(v,u)$ belongs to $A(D)$. Adding $e$ to $D$ yields a directed cycle $u \overset{e}{\rightarrow} v \overset{e'}{\rightarrow} u$, which is a strongly connected subdigraph of $D+e$. Therefore, $D$ is $\mathcal{D}_1$-saturated.
	
	\noindent{($\Rightarrow$)} Assume that $D$ is $\mathcal{D}_1$-saturated. Then $D$ contains no non-trivial strongly connected subdigraph. Therefore, $D$ is acyclic, so for any pair of distinct vertices $u,v\in V(D)$, at most one of the arcs in the set $\{(u,v),(v,u)\}$ belongs to $A(D)$. It remains to show that for any pair of distinct vertices $u,v\in V(D)$, exactly one of $(u, v)$ and $(v, u)$ belongs to $A(D)$, thereby implying that $D$ is an acyclic tournament.
	
	We may assume that there exist two distinct vertices $u,v\in V(D)$ such that neither $(u,v)$ nor \\$(v,u)$ is in $A(D)$. Since $D$ is $\mathcal{D}_1$-saturated, $D+(u,v)$ contains a directed cycle $C_1$. As $D$ is acyclic, $C_1$ must use the arc $(u,v)$. Consequently, $P_1=C_1-(u,v)$ is a $(v,u)$-dipath.
	
	Similarly, $D+(v,u)$ contains a directed cycle $C_2$. Since $D$ is acyclic, it follows that the arc $(v, u)$ must be in $C_2$, and so $P_2=C_2-(v,u)$ is a $(u,v)$-dipath.
	
	Since $P_1$ is a $(v,u)$-dipath and $P_2$ is a $(u,v)$-dipath, $D[V(P_1)\cup V(P_2)]$ contains a directed cycle, contradicting the acyclicity of $D$. Thus, $D$ must be an acyclic tournament.\hfill\qedsymbol
	
	The structure of $\mathcal{D}_1$-saturated digraphs is well-characterized. We next turn our attention to $\mathcal{D}_k$-saturated digraphs for $k \geq 2$. The following lemma determines the connectivity of such digraphs.
	
	\begin{lem}\label{l22}
		Let $k\geq 2$ and $n\geq k+1$ be integers. If $D$ is a $\mathcal{D}_k$-saturated digraph with $n$ vertices, then $\kappa(D)=k-1$.
	\end{lem}
	
	\noindent{\bf Proof.} Since $D$ is $\mathcal{D}_k$-saturated, it contains no $k$-strongly connected subdigraph, and so $\kappa(D) \leq k-1$. We only need to prove that $\kappa(D)\geq k-1$.
	
	On th contrary, assume that $\kappa(D)\leq k-2$. Let $S$ be a minimum separator of $D$. Then $|S|=\kappa(D)\leq k-2$. Let $B_1$ be a strong component of $D-S$, and let $B_2=D-(V(B_1)\cup S)$. Therefore, there exist vertices $v_1\in V(B_1)$ and $v_2\in V(B_2)$ such that either $(v_1,v_2)\notin A(D)$ or $(v_2,v_1)\notin A(D)$. 	
	
	Assume, without loss of generality, that $(v_1,v_2)\in A(D^c)$. Since $D$ is $\mathcal{D}_k$-saturated, there exists a digraph $D'\subseteq D+(v_1,v_2)$ such that $\kappa(D')\geq k$, and so $|V(D')|\geq k+1$. Since $D$ contains no $k$-strongly connected subdigraph, $D'$ cannot be a subdigraph of $D$ and $(v_1,v_2)\in A(D')$. Moreover, $V(D')\cap V(B_1)\neq \emptyset$ and $V(D')\cap V(B_2)\neq \emptyset$, so $V(D')\cap S$ forms a separator of $D'-(v_1,v_2)$.
	
	If $|V(B_1)\cap V(D')|=|V(B_2)\cap V(D')|=1$, then $|V(D')|\leq |S|+2\leq k$, contradicting $|V(D')|\geq k+1$. Thus, at least one of $|V(B_1)\cap V(D')|$ and $|V(B_2)\cap V(D')|$ is at least $2$. Without loss of generality, assume that $|V(B_1)\cap V(D')|\geq 2$. Define $S'=(V(D')\cap S)\cup \{v_1\}$. Since $|V(B_1)\cap V(D')|\geq 2$ and $V(D')\cap S$ is a separator of $D'-(v_1,v_2)$, $S'$ is a separator of $D'$. This yields a contradiction:
	$$k \leq \kappa(D') \leq |S'|\leq |S|+1 \leq k-1. $$
	This contradiction completes the proof. \hfill\qedsymbol
	
	Let $D$ be a non-complete $\mathcal{D}_k$-saturated digraph. By Lemma \ref{l22}, $\kappa(D)=k-1$. Since $D$ is not complete, it admits a minimum separator $S$ with $|S|=k-1$. Bang-Jensen et al. \cite{Bang2009} showed that the strong components of $D-S$ can be labeled $B_1,B_2,\ldots,B_t$ such that there is no arc from $B_j$ to $B_i$ whenever $j>i$. We say $B_1$ is the \emph{source strong component} of $D-S$ and $B_t$ is the \emph{sink strong component} of $D-S$. This ordering $(B_1,B_2,\ldots,B_t)$ is called an \emph{acyclic ordering} of the strong components of $D-S$.
	
	\begin{lem}\label{l23}
		Let $n$ and $k$ be integers such that $k\geq 2$ and $n\geq k+1$. Let $D$ be a $\mathcal{D}_k$-saturated digraph of order $n$. Let $S$ be a minimum separator of $D$, and let $(B_1,B_2,\ldots,B_t)$ denote an acyclic ordering of the strong components of $D-S$. Define $D'$ as the digraph with $V(D')=\bigcup_{i=1}^t V(B_i)~\text{and}~A(D')=\{(u,v)\in A(D): u\in V(B_i), v\in V(B_j)~\text{for all}~ 1\leq i < j\leq t\}.$
		Then $D'$ is a $t$-partite tournament.
	\end{lem}
	
	\noindent{\bf Proof.} Since $(B_1,B_2,\ldots, B_t)$ is an acyclic ordering of the strong components of $D-S$, there are no arcs from $B_j$ to $B_i$  in $D$ for any $j > i$. We claim that for all $1 \leq i < j \leq t$, every possible arc from $B_i$ to $B_j$ is present in $D$, and thus in $D'$. On the contrary, assume that there exist vertices $u \in B_i$ and $v \in B_j$ with $i<j$ such that $(u,v)\notin A(D)$. Then $(u,v)\in A(D^c)$, and by the $\mathcal{D}_k$-saturation of $D$, $D+(u,v)$ contains a $k$-strongly connected subdigraph $H$ and $(u,v) \in A(H)$. It is clear that $S \cap V(H)$ is a separator of $H$ satisfying $|S\cap V(H)| \leq |S|=k-1$, contradicting that $H$ is $k$-strongly connected. Hence, the arc $(u,v)$ must be present in $D$, and so $(u,v) \in A(D')$. This proves that $D'$ is a $t$-partite tournament. \hfill\qedsymbol
	
	Let $D$ be a non-complete $\mathcal{D}_k$-saturated digraph of order $n$. By Lemma \ref{l22}, $D$ contains a separator $S $ with $|S|=k-1$, and let $B_1$ be the source strong component of $D-S$ and $B_2=D-(S\cup V(B_1))$. We call $(S,D_1, D_2)$ a \emph{separation} of $D$, where $D_1=D[S\cup V(B_1)]$ and $D_2=D[S\cup V(B_2)]$.
	
	\begin{lem}\label{l24}
		Let $n$ and $k$ be integers with $k\geq 2$ and $n\geq k+1$. Let $D$ be a $\mathcal{D}_k$-saturated digraph. Assume $(S,D_1,D_2)$ is a separation of $D$, then each of the following must hold.
		\begin{itemize}
			\item[\rm{(i)}] If $e\in \bigcup_{i=1}^2 A(D_i^c)$, then any $k$-strongly connected subdigraph $D'$ of $D+e$ is contained in either $D_1+e$ or $D_2+e$. Furthermore, if $e\in A(D_{i}^c)\setminus A((D[S])^c)$, then $D'$ is a subdigraph of $D_i+e$, where $i=1,2$.
			\item[\rm{(ii)}] If $D[S]$ is a complete digraph, then both $D_1$ and $D_2$ are $\mathcal{D}_k$-saturated.
		\end{itemize}
	\end{lem}
	
	\noindent{\bf Proof.} Since $D$ is $\mathcal{D}_k$-saturated, for any arc $e\in A(D^c)$, the digraph $D+e$ contains a $k$-strongly connected subdigraph $D'$.
	
	\noindent{(i)} Let $e\in \bigcup_{i=1}^2 A(D_i^c)$ and $J=\{j\in\{1,2\}: V(D')\cap (V(D_j)-S)\neq \emptyset\}$. Suppose to the contrary that $|J|=2$, together with $e\in \bigcup_{i=1}^2 A(D_i^c)$, implies that $S\cap V(D')$ is a separator of $D'$. Hence, $$k\leq \kappa(D')\leq |S\cap V(D')|\leq |S|=k-1,$$ a contradiction. Therefore, we obtain that $|J|=1$ and $D'$ is a subdigraph of $D_j+e$ for $V(D')\cap (V(D_j)-S)\neq \emptyset$, where $j\in J$. If $e\in A(D_i^c)\setminus A((D[S])^c)$, then $V(D')\cap (V(D_i)-S)\neq \emptyset$ and $V(D')\cap (\bigcup_{j\neq i}V(D_j)-S)=\emptyset$, and so $D'$ is a subdigraph of $D_i+e$.
	
	\noindent{(ii)} Suppose $D[S]$ is a complete digraph. If $D_1$ is complete, then $D_1\cong K_k^*$, and so $D_1$ is $\mathcal{D}_k$-saturated. Assume that $D_1$ is not complete, every arc $e\in A(D_1^c)$ is incident with at least one vertex in $V(D_1)-S$, by (i), there exists a $k$-strongly connected subdigraph $D'$ of $D + e$ contained in $D_1 + e$. Since $D'$ is also a subdigraph of $D_1 + e$, and $e$ is arbitrary, $D_1$ is $\mathcal{D}_k$-saturated. By a similar analysis, we show that $D_2$ is also $\mathcal{D}_k$-saturated. This proves (ii).\hfill\qedsymbol
	
	\section{Saturation number of $\mathcal{D}_k$}\label{Se3}
	
	The \emph{join} of two digraphs $D_1$ and $D_2$, denoted by $D_1\vee D_2$, is the digraph with the vertex set $V(D_1 \vee D_2) = V(D_1)\cup V(D_2)$ and arc set $$A(D_1\vee D_2)=A(D_1)\cup A(D_2)\cup [V(D_1),V(D_2)].$$
	
	An undirected $k$-tree is any graph that is either $K_k$ or is obtained by joining a new vertex to a complete graph $K_k$ in a $k$-tree. A standard characterization of undirected $k$-trees can be found in \cite{Rose1974}. We now define the directed version of $k$-trees, which will be the main object of this section.
	
	\begin{Def}\label{d31}
		Let $n$ and $k$ be integers such that $k\geq 1$ and $n\geq k$. The digraph $T_{k}(n)$ is defined recursively as follows:
		\begin{itemize}
			\item[\rm{(i)}] $T_k(k)$ is the complete digraph $K_k^*$.
			\item[\rm{(ii)}] For $n > k$, given $T_k(n-1)$ and let $K$ be a complete subdigraph of $T_k(n-1)$ on $k$ vertices. Then $T_k(n)$ is the digraph with the vertex set $V(T_k(n)) = V(T_k(n-1)) \cup \{v_n\}$ and the arc set
			$$A(T_k(n)) = A(T_k(n-1)) \cup [v_n,V(K)] \cup (v_n, V(T_k(n-1)) \setminus V(K)),~ or$$
			$$A(T_k(n)) = A(T_k(n-1)) \cup [v_n,V(K)] \cup (V(T_k(n-1)) \setminus V(K),v_n).~~~~$$
		\end{itemize}
	\end{Def}
	
	The family $\mathcal{T}_{k-1}(n)$ consists of all directed $(k-1)$-trees with $n$ vertices. To show that every directed $(k-1)$-tree is $\mathcal{D}_k$-saturated, we now explore some properties of these digraphs.
	
	\begin{lem}\label{l32}
		Let $n$ and $k$ be integers with $k \geq 2$ and $n \geq k$. If $D \in \mathcal{T}_{k-1}(n)$, then $\delta^{\leftrightarrow}(D)=k-1$, $\kappa(D)=k-1$, and every minimum separator of $D$ is a $(k-1)$-clique.
	\end{lem}
	
	\noindent{\bf Proof.} We proceed by induction on $n$. For $n=k$, we have $D\cong K_k^*$, and so the lemma holds trivially.
	
	Now assume that $n \geq k+1$ and the lemma holds for each value smaller than $n$. By Definition \ref{d31}, $D$ is constructed from some $H\in\mathcal{T}_{k-1}(n-1)$ by adding a new vertex $v_n$, where $N_D^+(v_n) \cap N_D^-(v_n)$ forms a $(k-1)$-clique $K$ in $H$, and $v_n$ has unidirectional arcs to or from each vertex in $V(H) \setminus V(K)$. Without loss of generality, we assume that $v_n$ has only out-arcs to all vertices in $V(H) \setminus V(K)$.
	
	Firstly, we prove that $\delta^{\leftrightarrow}(D)=k-1$. By Definition \ref{d31}, we have $d_{D}^{\leftrightarrow}(v_n)=k-1$, so $\delta^{\leftrightarrow}(D)\leq d_{D}^{\leftrightarrow}(v_n)=k-1$. By the induction hypothesis, we have $\delta^{\leftrightarrow}(H)=k-1$, so $\delta^{\leftrightarrow}(D)\geq \min\{d_{D}^{\leftrightarrow}(v_n),\delta^{\leftrightarrow}(H)\}=k-1$. Thus, $\delta^{\leftrightarrow}(D)=k-1$.
	
	Next, we show that $\kappa(D)=k-1$. Since $V(K)$ is a separator of $D$, we have $\kappa(D) \leq |V(K)| = k-1$. It remains to prove that $\kappa(D) \ge k-1$. Assume that $S \subseteq V(D)$ is a separator of $D$ with $|S| \leq k-2$. We consider two cases.
	
	\noindent{\bf Case 1. $v_n\in S$.}
	
	Let $S'=S\setminus\{v_n\}$. Then $|S'|\leq k-3$ and $D-S=D-v_n-S'=H-S'$ is not strongly connected, which contradicts $\kappa(H)=k-1$.
	
	\noindent{\bf Case 2. $v_n\notin S$.}
	
	In this case, $S\subseteq V(H)$. Because $|S|\leq k-2<\kappa(H)=k-1$, $H-S$ is strongly connected. As $|S|\leq k-2$ and $|V(K)|=k-1$, there exists $w\in V(K)\setminus S$ such that $v_n$ has bidirectional arcs with $w$. Consequently, $D-S$ is strongly connected, also a contradiction. Therefore, $\kappa(D)=k-1$.
	
	Now we demonstrate that every minimum separator of $D$ is a $(k-1)$-clique. Let $S \subseteq V(D)$ be a minimum separator of $D$. Then $|S|=k-1$. Firstly, we show that $v_n\notin S$. If $v_n \in S$, then $S'=S \setminus\{v_n\}$ satisfies $|S'|=k-2$. Since $\kappa(H)=k-1$, we have $H-S'$ is still strongly connected, implying that $D-S=D-v_n-S'=H-S'$ is strongly connected, which contradicts that $S$ is a separator of $D$. Hence $v_n\notin S$, and so $S\subseteq V(H)$.
	
	If $S=V(K)$, then $S$ is a $(k-1)$-clique, and we are done. If $S\neq V(K)$, we claim that $S$ is a separator of $H$. Since $|S|=k-1$, $|K|=k-1$ and $S \neq V(K)$, we have $|V(K) \setminus S| \geq 1$. Assume, to the contrary, that $H-S$ is strongly connected. Pick any $w\in V(K)\setminus S$, then $v_n$ has bidirectional arcs with $w$. Thus, $D - S$ would be strongly connected, contradicting that $S$ is a separator of $D$. Therefore, $S$ is a separator of $H$ with $|S|=k-1$. By the induction hypothesis applied to $H$, $S$ is a $(k-1)$-clique. This completes the proof. \hfill\qedsymbol
	
	Based on the connectivity and degree properties established in Lemma \ref{l32}, we show that every directed $(k-1)$-tree contains at least two vertices with minimum reciprocal-degree $k-1$.
	
	\begin{lem}\label{l33}
		Let $k \geq 2$ and $n \geq k+1$ be integers. If $D \in \mathcal{T}_{k-1}(n)$, then there exists a vertex subset $U \subseteq V(D)$ such that:
		\begin{itemize}
			\item [\rm{(i)}] $|U| \geq 2$;
			\item [\rm{(ii)}] For every $u \in U$, $d_{D}^{\leftrightarrow}(u)=|N_D^+(u) \cap N_D^-(u)| = k-1$;
			\item [\rm{(iii)}] $D[U]$ is an acyclic tournament.
		\end{itemize}
	\end{lem}
	
	\noindent{\bf Proof.} We proceed by induction on $n$. If $n=k+1$, then $D$ is obtained from $K_{k}^*$ by adding a new vertex $v_{k+1}$ such that $v_{k+1}$ has bidirectional arcs with each vertex in a $(k-1)$-clique $K$ and has unidirectional arcs to or from $V(K_{k}^*)\setminus V(K)$. Let $K \subseteq K_k^*$ be a $(k-1)$-clique in $D$, and $v_k=V(K_k^*)\setminus V(K)$. Clearly, $d_{D}^{\leftrightarrow}(v_k)=d_{D}^{\leftrightarrow}(v_{k+1})=k-1$, and $D[\{v_k\}\cup \{v_{k+1}\}]\cong T_{2}$ is an acyclic tournament, i.e., $U=\{v_k,v_{k+1}\}$, the result holds trivially.
	
	Now assume $n \ge k+2$ and that the result holds for each value smaller than $n$. By Definition~\ref{d31}, $D$ is obtained from some $H\in\mathcal{T}_{k-1}(n-1)$ by adding a new vertex $v_n$, where $N_D^+(v_n) \cap N_D^-(v_n)$ forms a $(k-1)$-clique $K$ in $H$, which implies that $d_{D}^{\leftrightarrow}(v_n)= k-1$. By the induction hypothesis, $H$ contains a set $U'$ satisfying $(i)-(iii)$. Since $H[U']$ is an acyclic tournament and $K$ is a clique, $U'$ contains at most one vertex of $K$. If there exists a vertex $w\in U' \cap V(K)$, then $U=(U'\setminus \{w\})\cup \{v_n\}\subseteq V(D)$ satisfies conditions $(i)-(iii)$. Otherwise, $U=U'\cup \{v_{n}\}$ is a vertex subset satisfying the required properties. \hfill\qedsymbol
	
    Let $D$ be a directed $(k-1)$-tree and let $v$ be a vertex with $d_D^{\leftrightarrow}(v)=k-1$. Now we show that $D-v$ is also a directed $(k-1)$-tree.

	\begin{lem}\label{l34}
		Let $k\ge 2$ and $n\ge k$ be integers. Let $D \in \mathcal{T}_{k-1}(n)$ and $v \in V(D)$ satisfying $d_{D}^{\leftrightarrow}(v)= k-1$. Then $D-v \in \mathcal{T}_{k-1}(n-1)$.
	\end{lem}
	\noindent{\bf Proof.} We prove by induction on $n$. If $n=k$, then $D\cong K_k^*$. For any vertex $v\in V(D)$, we have $d_D^{\leftrightarrow}(v)=k-1$, and $D-v\cong K_{k-1}^*\in \mathcal{T}_{k-1}(k-1)$. Thus the result holds for $n=k$.
	
	Now consider $n=k+1$. By Definition \ref{d31}, $D$ is obtained from $K_{k}^*$ by adding a new vertex $v_{k+1}$ with bidirectional arcs to a $(k-1)$-clique $K$ and unidirectional arcs to or from the remaining vertex $v_{k}$, where $\{v_k\}=V(K_k^*)\setminus V(K)$. Then $d_{D}^{\leftrightarrow}(v_k)=d_{D}^{\leftrightarrow}(v_{k+1})=k-1$. Moreover, $D-v_k\cong D-v_{k+1}\cong K_{k}^* \in \mathcal{T}_{k-1}(k)$. Thus, the result holds for $n=k+1$.
	
	Assume that $n \geq k+2$ and the lemma is true for all integers $m$ satisfying $k \leq m <n$. By Definition \ref{d31}, there exists $v_n \in V(D)$ such that $H = D - v_n \in \mathcal{T}_{k-1}(n-1)$, $N_D^+(v_n) \cap N_D^-(v_n)$ is a $(k-1)$-clique $K$ in $H$, and $v_n$ has unidirectional arcs to or from the vertices in $V(H) \setminus V(K)$. Note that $d_D^{\leftrightarrow}(v_n)=|V(K)|=k-1$.
	
	If $v = v_n$, then $D-v=D - v_n = H \in \mathcal{T}_{k-1}(n-1)$, and we are done. Now suppose that $v\neq v_n$. We claim that $v\notin V(K)$. Suppose, to the contrary, that $v\in V(K)$, then $d_{H}^{\leftrightarrow}(v)=d_{D}^{\leftrightarrow}(v)-1=k-2$, which contradicts the fact that $\delta^{\leftrightarrow}(H)=k-1$ (by Lemma \ref{l32}). Thus, $v\notin V(K)$, and so $d_{H}^{\leftrightarrow}(v)= k-1$. By the inductive hypothesis, $H-v\in \mathcal{T}_{k-1}(n-2)$.
	
	Since $v \notin V(K)$, it follows that $K$ is a $(k-1)$-clique in $H-v$. Moreover, $v_n$ still has bidirectional arcs with all vertices in $V(K)$, and $v_n$ has unidirectional arcs with vertices in $(V(H) \setminus V(K)) \setminus \{v\}=V(H-v)\setminus V(K)$. Thus, $D - v$ is constructed by adding $v_n$ to $H-v\in \mathcal{T}_{k-1}(n-2)$ in the manner required by Definition \ref{d31}, so $D - v \in \mathcal{T}_{k-1}(n-1)$.  \hfill\qedsymbol
	
    Let $D$ be a non-complete directed $(k-1)$-tree. Let $S$ be a minimum separator of $D$, and let $(B_1,B_2,\ldots,B_t)$ be an acyclic ordering of the strong components of $D-S$. Next, we prove that each component $B_i$ contains at least one vertex with reciprocal-degree $k-1$ and that each $S$-lobe of $D$ is also a directed $(k-1)$-tree.
	
    \begin{lem}\label{l35}
		Let $n$ and $k$ be integers with $k \geq 2$ and $n \geq k+1$, and let $D \in \mathcal{T}_{k-1}(n)$. Let $S$ be a minimum separator of $D$, and let $(B_1,B_2, \ldots, B_t)$ be an acyclic ordering of the strong components of $D - S$. Then each component $B_i$ contains at least one vertex with reciprocal-degree $k-1$ in $D$.
	\end{lem}
	
	\noindent{\bf Proof.} We prove by induction on $n$. For $n=k+1$, $D$ is obtained from $K_k^*$ by adding a new vertex $v_{k+1}$. Let $K$ be a $(k-1)$-clique in $K_k^*$ and let $v_k$ be the unique vertex of $K_k^*$ not in $K$. Then $D$ has bidirectional arcs between $v_{k+1}$ and $V(K)$, and a unidirectional arc between $v_{k+1}$ and $v_k$. Hence $V(K)$ is the unique minimum separator of $D$, and the strong components of $D-V(K)$ are $\{v_k\}$ and $\{v_{k+1}\}$, each having reciprocal-degree $k-1$ in $D$.
	
	Now assume that $n\geq k+2$ and the result holds for all directed $(k-1)$-trees with fewer vertices. Let $S$ be a minimum separator of $D$. Then $|S|=k-1$. Let $v_n$ be the last vertex added in the construction of $D$ and let $H=D-v_n$. Without loss of generality, assume that there exists a $(k-1)$-clique $K$ in $H$ such that $N_{D}^+(v_n)\cap N_{D}^-(v_n)=V(K)$ and $N_{D}^+(v_n)=V(H)$. We claim that $v_n\notin S$. If $v_n\in S$, then $S'=S\setminus \{v_n\}$ is a separator of $H$ as $D-S=D-v_n-S'=H-S'$ is not strongly connected, contradicting the fact that $\kappa(H)=k-1$ (by Lemma \ref{l32}). Therefore, $v_n\notin S$, and so $v_n$ lies in a component of $D-S$, denoted by $B_l$. Since $d_D^{\leftrightarrow}(v_n)=k-1$, the component $B_l$ contains a vertex $v_n$ with reciprocal-degree $k-1$ in $D$.
	
	Now we consider any other component $B_i \neq B_l$ of $D-S$. By the definition of $D$, $\{v_n\}$ and $V(K)\setminus S$ belong to the same strong component of $D-S$, so $B_i\cap V(K)=\emptyset$.
	
	If $H-S$ is not strongly connected, then as $|S|=k-1$ and $\kappa(H)=k-1$, $S$ is also a minimum separator of $H$. Moreover, $B_i$ is a strong component of $H-S$. By the induction hypothesis, $B_i$ contains a vertex $w$ with $d_H^{\leftrightarrow}(w)=k-1$. Since $B_i\cap V(K)=\emptyset$, we have $w\notin V(K)$. Thus $d_D^{\leftrightarrow}(w)=d_H^{\leftrightarrow}(w)=k-1$ and the result holds.
	
	If $H-S$ is strongly connected, then $S=V(K)$ and $B_i=H-S$. By Lemma \ref{l34}, $H$ is a directed $(k-1)$-tree. Lemma \ref{l33} implies that $H$ contains at least two vertices of reciprocal-degree $k-1$, and these vertices form an acyclic tournament. Since $S$ is a clique, we have $B_i = H-S$ contains at least one vertex $w'$ with $d_H^{\leftrightarrow}(w')=k-1$. Since $B_i\cap V(K)=\emptyset$, we have $w'\notin V(K)$, and so $d_D^{\leftrightarrow}(w')=d_H^{\leftrightarrow}(w')=k-1$. Hence, $B_i$ contains a vertex $w'$ with reciprocal-degree $k-1$, completing the proof.  \hfill\qedsymbol
	
	\begin{lem} \label{l36}
		Let $k\ge 2$ and $n\ge k+1$ be integers, and let $D\in\mathcal{T}_{k-1}(n)$ be a directed $(k-1)$-tree. Let $S$ be a minimum separator of $D$, and let $(B_1,B_2,\ldots,B_t)$ be an acyclic ordering of the strong components of $D-S$. For any nonempty subset $J\subseteq\{1,2,\ldots,t\}$, we have $D_J=D[S\cup (\bigcup_{j\in J} V(B_j))]\in \mathcal{T}_{k-1}(n_J)$, where $n_J=|V(D_J)|$. In particular, if $|J|=1$, then we obtain $D_i=D[S\cup V(B_i)]\in \mathcal{T}_{k-1}(n_i)$, where $i\in J\subseteq \{1,2,\ldots, t\}$ and $n_i=|V(D_i)|$.
	\end{lem}
	
	\noindent{\bf Proof.} We proceed by induction on $n$. For $n=k+1$, by Definition \ref{d31}, $D$ is constructed by adding a new vertex $v_{k+1}$ to $K_k^*$, where $v_{k+1}$ has bidirectional arcs with a $(k-1)$-clique $K \subseteq K_k^*$ and unidirectional arcs to or from the unique vertex $v_k\in V(K_k^*)\setminus V(K)$. Clearly, $V(K)$ is the unique minimum separator of $D$ and $|V(K)|=k-1$. The strong components of $D\setminus V(K)$ are $\{v_k\}$ and $\{v_{k+1}\}$. Moreover, $D[V(K)\cup \{v_k\}]\cong D[V(K)\cup \{v_{k+1}\}]\cong K_k^*\in \mathcal{T}_{k-1}(k)$, so the result holds for $n=k+1$.
	
	Now assume that $n\geq k+2$ and the result holds for all directed $(k-1)$-trees with fewer vertices. Let $D\in \mathcal{T}_{k-1}(n)$. By Definition \ref{d31}, let $v_n$ be the last vertex added during the construction of $D$, and let $H=D\setminus v_n$. Let $K$ be the $(k-1)$-clique in $H$. Without loss of generality, assume that $N_{D}^+(v_n) \cap N_{D}^-(v_n)=V(K)$ and $N_{D}^+(v_n)=V(H)$.
	
	Let $S$ be a minimum separator of $D$. Then by Lemma \ref{l32}, we have $|S|=k-1$. We claim that $v_n\notin S$. Suppose $v_n\in S$. Then $S'=S\setminus\{v_n\}$ is a separator of $H$ because $D-S = H-S'$ is not strongly connected. This contradicts $\kappa(H)=k-1$ (by Lemma \ref{l32}). Hence $v_n\notin S$, and we have $S\subseteq V(H)$. If $H-S$ is strongly connected, then $S=V(K)$ and $D-S$ exactly contains two strong components $B_1=\{v_n\}$ and $B_2=H-S$. Therefore, $D_1=D[S\cup \{v_n\}]=K_{k}^*\in \mathcal{T}_{k-1}(k)$, and $D_2=D-v_n\in \mathcal{T}_{k-1}(n-1)$.
	
	If $H-S$ is not strongly connected, then as $|S|=k-1$ and $\kappa(H)=k-1$, $S$ is also a minimum separator of $H$ and $S$ is a separation clique of $H$. If $v_n\notin V(D_J)$, then $D_J\subseteq H$. By the inductive hypothesis, $D_J\in \mathcal{T}_{k-1}(n_J)$. If $v_n\in V(D_J)$, then $D_J-v_n\subseteq H$, and by the inductive hypothesis, $D_J-v_n\in \mathcal{T}_{k-1}(n_J-1)$. By the construction of $D$, $V(K)\setminus S$ and $\{v_n\}$ belong to the same strong component of $D-S$. Therefore, $V(K)\subseteq D_J-v_n$. Recall that $N_{D}^+(v_n) \cap N_{D}^-(v_n)=V(K)$ and $N_{D}^+(v_n)=V(H)$, so adding $v_n$ to $D_J-v_n$ satisfies the construction of directed $(k-1)$-tree. Thus, $D_J\in \mathcal{T}_{k-1}(n_J)$.
	
	For the special case when $|J|=1$, $D_i\in \mathcal{T}_{k-1}(n_i)$ holds trivially, where $i\in J\subseteq \{1,2,\ldots, t\}$, $D_i=D[S\cup V(B_i)]$ and $n_i=|V(D_i)|$. This completes the proof. \hfill\qedsymbol

	With these properties established, we now prove that every directed $(k-1)$-tree is $\mathcal{D}_k$-saturated.
	
	\begin{thm}\label{t37}
		Let $n$ and $k$ be integers with $k \geq 2$ and $n \geq k$. If $D\in \mathcal{T}_{k-1}(n)$, then $D$ is $\mathcal{D}_k$-saturated, and $|A(D)|=\binom{n-k+1}{2}+(k-1)(2n-k)$.
	\end{thm}
	\noindent{\bf Proof.} We proceed by induction on $n$. If $n=k$, then $D \cong K_{k}^*$ is $\mathcal{D}_k$-saturated and $|A(D)|=2\binom{k}{2}=\binom{k-k+1}{2}+(k-1)(2k-k)$.
	
	Now assume that $n \geq k+1$ and the theorem holds for all directed $(k-1)$-trees with fewer than $n$ vertices. Let $D$ be a directed $(k-1)$-tree on $n$ vertices. By Lemma \ref{l33}, we obtain that $D$ contains at least two vertices with reciprocal-degree $k-1$ and these vertices form an acyclic tournament. By Definition \ref{d31}, let $v_n$ be the last vertex added during the construction of $D$. Then $d_{D}^{\leftrightarrow}(v_n)=k-1$. By the inductive hypothesis, $D-v_n\in \mathcal{T}_{k-1}(n-1)$ is $\mathcal{D}_k$-saturated.
	
	Firstly, we prove that $D$ contains no $k$-strongly connected subdigraph. On the contrary, assume that there exists a $k$-strongly connected subdigraph $H \subseteq D$. If $H$ does not contain $v_n$, it would be a subdigraph of $D-v_n$, contradicting the $\mathcal{D}_k$-saturation of $D-v_n$. If $H$ contains $v_n$, then since $d_{H}^{\leftrightarrow}(v_n)\leq d_{D}^{\leftrightarrow}(v_n)=k-1$, removing $N_H^+(v_n) \cap N_H^-(v_n)$ would separate $v_n$ from $H$, contradicting that $H$ is $k$-strongly connected.
	
	Next, we will prove that $D+e$ contains a $k$-strongly connected subdigraph for any $e \in A(D^c)$. If $e$ is not incident to $v_n$, then as $D-v_n$ is $\mathcal{D}_k$-saturated, $(D-v_n)+e$ contains a $k$-strongly connected subdigraph, which is also a subdigraph of $D+e$. Now suppose $e$ is incident to $v_n$, say $e = (u,v_n)$ with $u \notin N_D^-(v_n)$. By Definition \ref{d31}, $(v_n,u) \in A(D)$.
	
	If there exists $w \notin \{u,v_n\}$ with $d_{D}^{\leftrightarrow}(w)=k-1$, then by Lemma \ref{l34}, $D-w \in \mathcal{T}_{k-1}(n-1)$ is $\mathcal{D}_k$-saturated and $e$ is still missing in $D-w$, so $(D-w)+e$ contains a $k$-strong subdigraph, which is also in $D+e$. Thus we may assume that $D$ has exactly two vertices $u$ and $v_n$ satisfying $d_{D}^{\leftrightarrow}(u)=d_{D}^{\leftrightarrow}(v_n)=k-1$.
	
	 Now we show that $D+(u,v_n)$ contains a $k$-strongly connected subdigraph. Let $S$ be any vertex set of size $k-1$ in $D + e$. If $S$ is not a $(k-1)$-clique, then by Lemma~\ref{l32}, $D - S$ is strongly connected, so $(D + e) - S$ is also strongly connected.
	
	If $S$ is a $(k-1)$-clique, then by Lemma~\ref{l35} and Lemma~\ref{l36}, each strong component of $D-S$ contains a vertex with reciprocal-degree $k-1$, and each $S$-lobe of $D$ is a directed $(k-1)$-tree. Consequently, there are only two $S$-lobes in $D$, one containing $u$ and the other containing $v_n$. Thus $S$ is not a separator of $D+(u,v_n)$, and we conclude that $D+(u,v_n)$ is $k$-strongly connected. Therefore, $D$ is $\mathcal{D}_k$-saturated.
	
	By Definition \ref{d31}, $D \in \mathcal{T}_{k-1}(n)$ is constructed by adding $v_n$ to $D'=D-v_n \in \mathcal{T}_{k-1}(n-1)$. By inductive hypothesis, $|A(D')|=\binom{n-k}{2}+(k-1)(2n-k-2)$. Thus, \begin{flalign*}
		|A(D)|&=|A(D')|+|[\{v_n\}, V(K)]|+|(\{v_n\},V(D')\setminus V(K))|\\
		&=\binom{n-k}{2}+(k-1)(2n-k-2)+2(k-1)+(n-k)\\
		&=\binom{n-k+1}{2}+(k-1)(2n-k).
	\end{flalign*}
	This completes the proof.
	\hfill\qedsymbol
	
	The following observation will be useful in the proof of the main result regarding the saturation number.
	
	\begin{Obs}\label{o38}
		Let $a$ and $b$ be two positive integers. If $a, b\geq 1$, then $\binom{a+b}{2} =\binom{a}{2}+\binom{b}{2}+ab$.
	\end{Obs}
	
	We now state and prove the main theorem of this section, which determines the saturation number for $\mathcal{D}_k$.
	
	\begin{thm}\label{t39}
		Let $k\geq 1$ and $n\geq k$ be integers. Then $$\mathrm{sat}(n,\mathcal{D}_k)=\binom{n-k+1}{2}+(k-1)(2n-k).$$
	\end{thm}
	
	\noindent{\bf Proof.} For $k=1$, the statement follows directly from Lemma~\ref{l21}, since each $\mathcal{D}_1$-saturated digraph $D$ is an acyclic tournament and has exactly $\binom{n}{2} = \binom{n-1+1}{2}+(1-1)(2n-1)$ arcs. Assume $k\geq 2$ in the following proof.
	
    The inequality $\mathrm{sat}(n,\mathcal{D}_k) \leq \binom{n-k+1}{2}+(k-1)(2n-k)$ is provided by Theorem~\ref{t37}, as directed $(k-1)$-tree achieves this bound and is $\mathcal{D}_k$-saturated.
	
	We will prove $\mathrm{sat}(n,\mathcal{D}_k) \geq \binom{n-k+1}{2}+(k-1)(2n-k)$ by induction on $n$. If $n=k$, the unique $n$-vertex $\mathcal{D}_k$-saturated digraph is $K_{k}^*$, and so the result holds.
	
	Let $D$ be a $\mathcal{D}_k$-saturated digraph of order $n \geq k+1$. We assume that the result holds for each value smaller than $n$. Since $n\geq k+1$, $D$ is not a complete digraph. Let $S$ be the minimum separator of $D$, $B_1$ be the source strong component of $D-S$, and $B_2=D-(S\cup V(B_1))$. By Lemma \ref{l22}, $|S|=k-1$. Let $D_1=D[S\cup V(B_1)]$, $D_2=D[S\cup V(B_2)]$, $|V(B_1)|=b_1$ and $|V(B_2)|=b_2$. Then $b_1+b_2=n-k+1$.
	
	If $D[S]\cong K_{k-1}^*$ is a clique, then by Lemma \ref{l24}, both $D_1$ and $D_2$ are $\mathcal{D}_k$-saturated. It follows by induction that both $|A(D_1)|\geq \binom{b_1}{2}+(k-1)(2(b_1+k-1)-k)$ and $|A(D_2)|\geq \binom{b_2}{2}+(k-1)(2(b_2+k-1)-k)$. Thus,
	\begin{flalign*}
		|A(D)|=&|A(D_1)|+|A(D_2)|-|A(D[S])|+|[V(B_1),V(B_2)]_D|\\
		\geq& \binom{b_1}{2}+(k-1)(2(b_1+k-1)-k)+\binom{b_2}{2}+(k-1)(2(b_2+k-1)-k)-2\binom{k-1}{2}+b_1 b_2\\
		=& \binom{b_1}{2}+\binom{b_2}{2}+b_1 b_2+(k-1)(2(b_1+b_2+2k-2)-2k-k+2)\\
		=&\binom{n-k+1}{2}+(k-1)(2n-k).&
	\end{flalign*}
	
	Now we assume that $D[S]$ is not a $(k-1)$-clique and $|A(D[S])|=m$. Since $D$ is $\mathcal{D}_k$-saturated, it follows that for any arc $e\in A((D[S])^c)$, there exists a $k$-strongly connected subdigraph $D'\subseteq D+e$. By Lemma \ref{l24}, $D'$ is either a subdigraph of $D_1+e$ or a subdigraph of $D_2+e$. If $D_1+e$ contains no $k$-strongly connected subdigraph, then modify $D_1$ by adding the arc $e$. Similarly, if $D_2+e$ contains no $k$-strongly connected subdigraph, then modify $D_2$ by adding the arc $e$. Repeat this operation for every arc $e\in A((D[S])^c)$. Thus, with the addition of at most $2\binom{k-1}{2}-m$ arcs, we obtain two $\mathcal{D}_k$-saturated digraphs $D_1'$ and $D_2'$. It concludes by induction that both $|A(D_1')|\geq \binom{b_1}{2}+(k-1)(2(b_1+k-1)-k)$ and $|A(D_2')|\geq \binom{b_2}{2}+(k-1)(2(b_2+k-1)-k)$. Thus
	\begin{flalign*}
		|A(D)|\geq&|A(D_1')|+|A(D_2')|-2\binom{k-1}{2}+|[V(B_1),V(B_2)]_D|\\
		\geq& \binom{b_1}{2}+(k-1)(2(b_1+k-1)-k)+\binom{b_2}{2}+(k-1)(2(b_2+k-1)-k)-2\binom{k-1}{2}+b_1 b_2\\
		=&\binom{n-k+1}{2}+(k-1)(2n-k).&
	\end{flalign*}
This completes the proof.
	\hfill\qedsymbol
	
	\section{Extremal number of $\mathcal{D}_k$} \label{Se4}
	
    In this section, we aim to determine the maximum number of arcs in an $n$-vertex $\mathcal{D}_k$-saturated digraph. Firstly, we define the following family of digraphs.

	\begin{Def}\label{d41}
		Let $k \geq 2$ and $n \geq 2(k-1)$ be integers. Suppose that $n=t(k-1)+r$, where $0\leq r<k-1$.  The family of digraphs $\mathcal{D}_U(n, k)$ is defined as follows. A digraph $D\in\mathcal{D}_U(n, k)$ if its vertex set can be partitioned into $t+1$ disjoint subsets $V_0, V_1, \ldots, V_t$ satisfying the following conditions:
		\begin{itemize}
			\item[\rm{(i)}] $|V_0|=|V_1|=\ldots=|V_{t-1}|=k-1$ and $|V_t|=r$;
			
			\item[\rm{(ii)}] $D[V_0]$ is a transitive tournament;
			
			\item[\rm{(iii)}] For each $i \in \{1,2, \ldots, t\}$, the induced subdigraph $D[V_i]$ is a complete digraph;
			
			\item[\rm{(iv)}] All arcs between $V_0$ and $\bigcup_{i=1}^t V_i$ are present in both directions. That is, for any $u\in V_0$ and $v\in \cup_{i=1}^t V_i$, both $(u,v)$ and $(v,u)$ are arcs in $D$;
			
			\item[\rm{(v)}] $D[\bigcup_{i=1}^t V_i]\setminus \bigcup_{i=1}^t A(D[V_i])$ forms a $t$-partite transitive tournament $T_{|V_1|,\ldots,|V_t|}$;
			
			\item[\rm{(vi)}] There are no other arcs in $D$.
		\end{itemize}
	\end{Def}
	
	To confirm the validity of $\mathcal{D}_U(n,k)$ as a candidate for the extremal construction, we initially verify that for any $D \in \mathcal{D}_U(n,k)$, the digraph $D$ is $\mathcal{D}_k$-saturated.
	
	\begin{thm}\label{t42}
		Let $k\geq 2$ and $n\geq 2(k-1)$ be integers. If $D \in \mathcal{D}_U(n,k)$, then
		\begin{itemize}
			\item[\rm{(i)}] $D$ is $\mathcal{D}_k$-saturated;
			
			\item[\rm{(ii)}] $|A(D)|\leq \frac{3}{2}(k-\frac{4}{3})(n-k+1)+\binom{n}{2}$, with equality if and only if $n$ is a multiple of $k-1$.
		\end{itemize}
	\end{thm}
	\noindent{\bf Proof.} (i) Firstly, we show that $D$ contains no $k$-strongly connected subdigraph. By Definition \ref{d41}, $V_0$ is a separator of $D$ with $|V_0|=k-1$, so $\kappa(D) \leq k-1 < k$. Now consider any subdigraph $H \subseteq D$.
	
	Assume there exists exactly one $i \in \{1,2, \ldots, t\}$ such that $V(H) \cap V_i\neq\emptyset$. If $V(H) \cap V_0\neq \emptyset$, then $V(H) \cap V_i$ is a separator of $H$. By $|V(H) \cap V_i| \leq |V_i| \leq k-1$, we know $H$ is not $k$-strongly connected. Otherwise, $H$ is not $k$-strongly connected as $|V(H)|\leq |V_i|\leq k-1$.
	
	Now suppose that there exist at least two distinct indices $i, j \in \{1,2, \ldots, t\}$ such that $V(H) \cap V_i \neq \emptyset$ and $V(H) \cap V_j \neq \emptyset$.	If $V(H)\cap V_0 \neq \emptyset$, then $V_0 \cap V(H)$ is a separator of $H$. By $|V_0 \cap V(H)|\leq k-1$, we know $\kappa(H) <k$.
	If $V(H) \subseteq \bigcup_{i=1}^t V_i$, then by Definition \ref{d41}, $H$ can not be strongly connected. Therefore, $D$ contains no $k$-strongly connected subdigraph.
	
	It remains to show that for any arc $e\in A(D^{c})$, the digraph $D+e$ contains a $k$-strongly connected subdigraph. It is clear that $A(D^c) \neq \emptyset$. Now we consider the following two cases.
	
	\noindent{\bf Case 1.} $e=(u,v)\in A(D^c[V_0])$.
	
	Since $D[V_0]$ is a transitive tournament and $(u,v)\notin A(D)$, the reverse arc $(v, u)$ is present in $D$. Thus, for every $i\in \{1,2,\ldots,t-1\}$, $D'=D[V_i\cup\{u,v\}]\subseteq D+(u,v)$ forms a complete subdigraph $K_{k+1}^*$, which is $k$-strongly connected. Hence, $D+e$ contains a $k$-strongly connected subdigraph.
	
	\noindent{\bf Case 2.} $e=(v_j,v_i)\in [V_i,V_j]_{D^c}$ for some $1\leq i<j\leq t$.
	
	By Definition~\ref{d41}, $D[\bigcup_{i=1}^t V_i]\setminus \bigcup_{i=1}^t A(D_i)$ forms a $t$-partite transitive tournament $T_{|V_1|,\ldots,|V_t|}$, which implies that $(v_i,v_j)\in A(D)$. Let $D'' = D[V_0 \cup V_i \cup V_j] + e$, we will prove that $D''$ is $k$-strongly connected.
	
	Let $S \subseteq V(D'')$ with $|S| \leq k-1$. We show that $D'' - S$ is strongly connected. If $S=V_0$, then $D''-S$ is strongly connected as $D[V_i]$ and $D[V_j]$ are strongly connected in $D''-S$ and $[v_i,v_j]\in A(D''-S)$. Otherwise, by Definition~\ref{d41}, it is clear that for any two vertices $u,v\in V(D''-S)$, there exist both $(u, v)$-dipath and $(v, u)$-dipath. Thus, $D'' - S$ is strongly connected, and $D''$ is $k$-strongly connected. This completes the proof that $D$ is $\mathcal{D}_k$-saturated.
	
	(ii) By a simple calculation, we have $|A(D)|\leq \frac{3}{2}(k-\frac{4}{3})(n-k+1)+\binom{n}{2}$, where the equality holds if and only if $n$ is a multiple of $k-1$.\hfill\qedsymbol
	
    Theorem \ref{t42} provides a class of $\mathcal{D}_k$-saturated digraphs, which leads us to conjecture that these digraphs attain the extremal number for sufficiently large $n$.
	
	\begin{conj}
		For sufficiently large $n$, $$\mathrm{ex}(n,\mathcal{D}_k)= \frac{3}{2}(k-\frac{4}{3})(n-k+1)+\binom{n}{2}.$$
	\end{conj}
	
	 Since every $\mathcal{D}_k$-saturated digraph contains no $k$-strongly connected subdigraph, it suffices to derive an upper bound on the number of arcs in an $n$-vertex digraph without $k$-strongly connected subdigraph. Firstly, we present a preliminary upper bound in the following lemma.
	
	\begin{lem}\label{l43}
		Let $k\geq 2$ and $n\geq k$ be integers. If $D$ is a digraph of order $n$ that contains no $k$-strongly connected subdigraph, then
		$$|A(D)|\leq 2\binom{n}{2}-\frac{(n-k+1)^2-1}{3}.$$
	\end{lem}
	
	\noindent{\bf Proof.} We proceed by induction on $n$. Let $D$ be an $n$-vertex digraph that contains no $k$-strongly connected subdigraph. If $n=k$, then $|A(D)|\leq 2\binom{k}{2}= 2\binom{k}{2}-\frac{(k-k+1)^2-1}{3}$.
	
	Now we consider the case when $n\geq k+1$. Assume the lemma holds for all digraphs of order less than $n$ that contain no $k$-strongly connected subdigraph. Since $D$ contains no $k$-strongly connected subdigraph, it contains a separator $S$ of $D$ with $|S| = k-1$. Let $B_1$ be the source strong component of $D-S$ and $B_2=D-(S\cup V(B_1))$. Clearly, $D$ misses at least $|B_1||B_2|$ arcs between $B_1$ and $B_2$. By the induction hypothesis, $D_2=D[S\cup V(B_2)]$ misses at least $\frac{|B_2|^2-1}{3}$ arcs. Hence,
	\begin{flalign*}
		|A(D)|\leq& 2\binom{n}{2}-|B_1||B_2|-\frac{|B_2|^2-1}{3}\\
		\leq &2\binom{n}{2}-\frac{|B_1|^2+2|B_1||B_2|+|B_2|^2-1}{3}\\
		=&2\binom{n}{2}-\frac{(n-k+1)^2-1}{3}.
	\end{flalign*}
	This completes the proof. \hfill\qedsymbol
	
	Lemma \ref{l43} provides a valid upper bound, it is not tight enough to match the arc number of \(\mathcal{D}_U(n,k)\) in Theorem \ref{t42}(ii). To narrow this gap, we prove a stronger upper bound, which is closer to the candidate extremal value.
	
	\begin{thm}
		Let $k\geq 2$ and $n \geq 3(k-1)$ be integers. If $D$ is a digraph of order $n$ that contains no $k$-strongly connected subdigraph, then $|A(D)|\leq\binom{n-k+1}{2}+\frac{17}{6}(k-1)(n-k+1).$
	\end{thm}
	
	\noindent{\bf Proof.} We proceed by induction on $n$. Let $D$ be an $n$-vertex digraph that contains no $k$-strongly connected subdigraph. Let $n = \alpha (k-1)$, where $\alpha \geq 3$ since $n\ge 3(k-1)$. We consider the following two cases based on the value of $\alpha$.
	
	\noindent{\bf Case 1.} $3 \leq \alpha \leq 4.$
	
	By Lemma \ref{l43}, we have $|A(D)|\leq 2\binom{n}{2}-\frac{(n-k+1)^2-1}{3}$. It suffices to show that
	$$2\binom{n}{2}-\frac{(n-k+1)^2-1}{3}\leq\binom{n-k+1}{2}+\frac{17}{6}(k-1)(n-k+1),$$
	which is equivalent to
	\begin{flalign*}
		&2\binom{n}{2}-\frac{(n-k+1)^2-1}{3}-\binom{n-k+1}{2}-\frac{17}{6}(k-1)(n-k+1)\\
		=&\frac{1}{6}(k-1)^2(\alpha^2-7\alpha+12)-\frac{1}{2}(\alpha+1)(k-1)+\frac{1}{3}\\
		\leq &\frac{1}{6}(k-1)^2(\alpha^2-7\alpha+12)\leq 0.
	\end{flalign*}
	The function $f(\alpha)=\alpha^2-7\alpha+12=(\alpha-3)(\alpha-4)$ is non-positive for $3 \leq \alpha \leq 4$. Hence, the result holds in this case.
	
	\noindent{\bf Case 2.} $\alpha > 4$.
	
	Assume the result holds for each value smaller than $n$. Since $D$ contains no $k$-strongly connected subdigraph, it contains a separator $S$ of $D$ with $|S| = k-1$. Let $B_1$ be the source strong component of $D-S$, and let $B_2=D-(S\cup V(B_1))$. Set $|B_1| =b_1 = \beta_1(k-1)$, and $|B_2| = b_2 = \beta_2(k-1)$. Then $b_1 + b_2 = n - k + 1$. We may assume $b_1 \leq b_2$, so that $\beta_1\leq \beta_2$. Define $D_1=D[S\cup V(B_1)]$ and $D_2=D[S\cup V(B_2)]$. We divide into three subcases based on the value of $\beta_1$.
	
	\noindent{\bf Subcase 2.1.} $\beta_1 \leq 1$.
	
	In this subcase, $\beta_2+1=\alpha-\beta_1>3$, so the induction hypothesis holds for $D_2$. Thus,
	\begin{flalign*}
		|A(D)|=&|A(D[V(B_1)])|+|[V(B_1),S]_D|+|A(D_2)|+|[V(B_1),V(B_2)]_D|\\
		\leq & 2\binom{b_1}{2}+2b_1(k-1)+\binom{b_2}{2}+\frac{17}{6}(k-1)b_2+b_1 b_2\\
		=&\binom{b_1}{2}+\binom{b_2}{2}+b_1 b_2+\frac{b_1^2-b_1}{2}+2b_1(k-1)+\frac{17}{6}(k-1)b_2.
	\end{flalign*}
	Since $0<\beta_1\leq 1$, we have
	\begin{equation}\label{eq1}
		\begin{aligned}
		&\frac{b_1^2-b_1}{2}+2b_1(k-1)-\frac{17}{6}(k-1)b_1\\
		= &\frac{\beta_1^2 (k-1)^2-\beta_1(k-1)}{2}-\frac{5}{6}\beta_1(k-1)^2\\
		= &\frac{3\beta_1^2-5 \beta_1}{6}(k-1)^2-\frac{\beta_1(k-1)}{2}\\
		\leq & \frac{\beta_1(\beta_1-\frac{5}{3})}{2}(k-1)^2\leq 0.
	\end{aligned}
	\end{equation}
	Combining Observation \ref{o38} and the inequality \eqref{eq1}, we have
	\begin{flalign*}
		|A(D)|\leq & \binom{b_1+b_2}{2}+\frac{17}{6}(k-1)(b_1+b_2)\\
		= &\binom{n-k+1}{2}+\frac{17}{6}(k-1)(n-k+1).
	\end{flalign*}
	\noindent{\bf Subcase 2.2.} $\beta_1 \geq 2$.
	
	Since $\beta_2\geq \beta_1$, we have $b_2+k-1\geq b_1+k-1=(\beta_1+1)(k+1) \geq 3(k+1)$. In this subcase, both $D_1$ and $D_2$ have at least $3(k-1)$ vertices, so by the induction hypothesis, we obtain that
	\begin{flalign*}
		|A(D)|&=|A(D_1)|+|A(D_2)|-|A(D[S])|+|[V(B_1), V(B_2)]_D|\\
		&\leq \binom{b_1}{2}+\frac{17}{6}(k-1)b_1+\binom{b_2}{2}+\frac{17}{6}(k-1)b_2+b_1 b_2\\
		&=\binom{b_1+b_2}{2}+\frac{17}{6}(k-1)(b_1+b_2)\\
		&=\binom{n-k+1}{2}+\frac{17}{6}(k-1)(n-k+1).
	\end{flalign*}
	
	\noindent{\bf Subcase 2.3.} $1< \beta_1 < 2$.
	
	In this subcase, $\beta_2=\alpha-\beta_1-1 > 1$. If $\beta_2\geq 2$, then $D_2$ satisfies the induction hypothesis. By Lemma \ref{l43}, $|A(D[V(B_1)])|\leq 2\binom{b_1}{2}-\frac{(b_1-k+1)^2-1}{3}$. Since $1< \beta_1 < 2$, we have
	\begin{equation}\label{eq2}
		\begin{aligned}
			&\binom{b_1}{2}-\frac{(b_1-k+1)^2-1}{3}+2b_1(k-1)- \frac{17}{6}(k-1)b_1\\
			= &\frac{\beta_1^2- \beta_1-2}{6}(k-1)^2-\frac{\beta_1(k-1)}{2}+\frac{1}{3}\\
			\leq & \frac{(\beta_1+1)(\beta_1-2)}{6}(k-1)^2\leq 0.
		\end{aligned}
	\end{equation}
	 Combining the inequality \eqref{eq2} and the induction hypothesis for $D_2$, we obtain that
	\begin{flalign*}
		|A(D)| & = |A(D[V(B_1)])|+|[V(B_1),S]_D|+|A(D_2)|+|[V(B_1),V(B_2)]_D|\\
		& \leq 2\binom{b_1}{2}-\frac{(b_1-k+1)^2-1}{3}+2b_1(k-1)+\binom{b_2}{2}+\frac{17}{6}(k-1)b_2+b_1 b_2\\
		& =\binom{b_1}{2}+\binom{b_2}{2}+b_1 b_2+\binom{b_1}{2}-\frac{(b_1-k+1)^2-1}{3}+2b_1(k-1)+\frac{17}{6}(k-1)b_2\\
		&\leq \binom{b_1+b_2}{2}+\frac{17}{6}(k-1)(b_1+b_2)\\
		&=\binom{n-k+1}{2}+\frac{17}{6}(k-1)(n-k+1).
	\end{flalign*}
	If $1 < \beta_2 < 2$, then we use Lemma \ref{l43} to $D[V(B_1)]$ and $D_2$, we obtain that
	\begin{flalign*}
		|A(D)| & = |A(D[V(B_1)])|+|[V(B_1),S]_D|+|A(D_2)|+|[V(B_1),V(B_2)]_D|\\
		& \leq 2\binom{b_1}{2}-\frac{(b_1-k+1)^2-1}{3}+2b_1(k-1)+2\binom{b_2+k-1}{2}-\frac{b_2^2-1}{3}+b_1b_2\\
		& = (k-1)^2(\frac{2}{3}\beta_1^2+\frac{8}{3}\beta_1+\frac{2}{3}\beta_2^2+2\beta_2+\beta_1\beta_2+\frac{2}{3})-(\beta_1+\beta_2+1)(k-1).
	\end{flalign*}
	Since \begin{flalign*}
		&\frac{2}{3}\beta_1^2+\frac{8}{3}\beta_1+\frac{2}{3}\beta_2^2+2\beta_2+\beta_1\beta_2+\frac{2}{3}- \frac{(\beta_1+\beta_2)^2}{2}-\frac{17}{6}(\beta_1+\beta_2)\\
		=&\frac{1}{6}\beta_1^2-\frac{1}{6}\beta_1+\frac{1}{6}\beta_2^2-\frac{5}{6}\beta_2+\frac{2}{3}\\
		\leq & \frac{1}{6}\beta_1^2-\frac{1}{6}\beta_1+\frac{1}{6}\beta_1^2-\frac{5}{6}\beta_1+\frac{2}{3}\\
		=&\frac{1}{3}\beta_1(\beta_1-3)+\frac{2}{3}\\
		\leq& 0
		\end{flalign*}
	holds for $1<\beta_1\leq \beta_2 <2$, we have \begin{flalign*}
		|A(D)|\leq & (k-1)^2(\frac{2}{3}\beta_1^2+\frac{8}{3}\beta_1+\frac{2}{3}\beta_2^2+2\beta_2+\beta_1\beta_2+1)-(\beta_1+\beta_2+1)(k-1)\\
		\leq & (k-1)^2(\frac{(\beta_1+\beta_2)^2}{2}+\frac{17}{6}(\beta_1+\beta_2))-\frac{\beta_1+\beta_2}{2}(k-1)\\
		=&\binom{b_1+b_2}{2}+\frac{17}{6}(k-1)(b_1+b_2)\\
		=&\binom{n-k+1}{2}+\frac{17}{6}(k-1)(n-k+1),
	\end{flalign*} which completes the proof.\hfill\qedsymbol


\vspace{1cm}
	
	
\begin{thebibliography}{20}
	\bibitem{Anderson2016}
	J. Anderson, H.-J. Lai, X. X. Lin, M. R. Xu,
	\textit{On $k$-maximal strength digraphs},
	\textit{Journal of Graph Theory}
	\textbf{84}(1) (2016) 17--25.
	
	\bibitem{Bang2009}
	J. Bang-Jensen, G. Gutin,
	\textit{Digraphs: Theory, Algorithms and Applications},
	2nd edn. Springer-Verlag, London, 2009.
	
	\bibitem{Bernshteyn2016}
	A. Bernshteyn, A. Kostochka,
	\textit{On the number of edges in a graph with no $(k+1)$-connected subgraphs},
	\textit{Discrete Mathematics}
	\textbf{339}(2) (2016) 682--688.
	
	\bibitem{Erdos1964}
	P. Erd\H{o}s, A. Hajnal, J. W. Moon,
	\textit{A problem in graph theory},
	\textit{The American Mathematical Monthly}
	\textbf{71}(10) (1964) 1107--1110.
	
	\bibitem{Lai1990}
	H.-J. Lai,
	\textit{The size of strength-maximal graphs},
	\textit{Journal of Graph Theory}
	\textbf{14}(2) (1990) 187--197.
	
	\bibitem{Lei2019}
	H. Lei, S. O, Y. T. Shi, D. B. West, X. D. Zhu,
	\textit{Extremal problems on saturation for the family of $k$-edge-connected graphs},
	\textit{Discrete Applied Mathematics}
	\textbf{260} (2019) 278--283.
	
	\bibitem{Lin2016}
	X. X. Lin, S. H. Fan, H.-J. Lai, M. R. Xu,
	\textit{On the lower bound of $k$-maximal digraphs},
	\textit{Discrete Mathematics}
	\textbf{339}(10) (2016) 2500--2510.
	
	\bibitem{Mader1971}
	W. Mader,
	\textit{Minimale $n$-fach kantenzusammenh\"{a}ngende Graphen},
	\textit{Mathematische Annalen}
	\textbf{191} (1971) 21--28.
	
	\bibitem{Mader1972}
	W. Mader,
	\textit{Existenz $n$-fach zusammenh\"{a}ngender Teilgraphen in Graphen gen\"{u}gend gro\ss en Kantendichte},
	\textit{Abhandlungen aus dem Mathematischen Seminar der Universit\"{a}t Hamburg}
	\textbf{37} (1972) 86--97.
	
	\bibitem{Mader1979}
	W. Mader,
	\textit{Connectivity and edge-connectivity in finite graphs},
	In: B. Bollob\'{a}s (ed.),
	\textit{Surveys in Combinatorics},
	London Mathematical Society Lecture Note Series, Vol. 38,
	Cambridge University Press, Cambridge, 1979, pp. 66--95.
	
	\bibitem{Rose1974}
	D. J. Rose,
	\textit{On simple characterizations of $k$-trees},
	\textit{Discrete Mathematics}
	\textbf{7} (1974) 317--322.
	
	\bibitem{Turan1941}
	P. Tur\'{a}n,
	\textit{Eine Extremalaufgabe aus der Graphentheorie},
	\textit{Matematikai \'{e}s Fizikai Lapok}
	\textbf{48} (1941) 436--452.
	
	\bibitem{Wenger2014}
	P. Wenger,
	\textit{A note on the saturation number of the family of $k$-connected graphs},
	\textit{Discrete Mathematics}
	\textbf{323} (2014) 81--83.
	
	\bibitem{Xu2020}
	L. Q. Xu, H.-J. Lai, Y. Z. Tian,
	\textit{On the extremal sizes of maximal graphs without $(k+1)$-connected subgraphs},
	\textit{Discrete Applied Mathematics}
	\textbf{285} (2020) 397--406.
	
	\bibitem{Yuster2003}
	R. Yuster,
	\textit{A note on graphs without $k$-connected subgraphs},
	\textit{Ars Combinatoria}
	\textbf{67} (2003) 231--235.
\end{thebibliography}
\end{document}